\documentclass[twoside,centertags]{amsart}
\usepackage{amsmath,amssymb,amsthm, latexsym}
\usepackage{mathrsfs}
\usepackage{xypic}

\newtheorem*{theoremA}{Theorem A}
\newtheorem*{theoremB}{Theorem B}

\newtheorem{theorem}{Theorem} [section]
\newtheorem{lemma}[theorem]{Lemma}
\newtheorem{corollary}[theorem]{Corollary}
\newtheorem{proposition}[theorem]{Proposition}

\theoremstyle{definition}

\newtheorem{remark}[theorem]{Remark}

\numberwithin{equation}{section}

\newcommand{\R}{R\{-\}}
\newcommand{\RR}{\mathcal{R}\{-\}}
\newcommand{\sCoalg}{\mathrm{sCoalg}}
\newcommand{\Coalg}{\mathrm{Coalg}}
\newcommand{\sMod}{\mathrm{sMod}}
\newcommand{\Mod}{\mathrm{Mod}}

\newcommand{\sSet}{\mathrm{sSet}}
\newcommand{\inj}{\mathrm{inj}}

\newcommand{\PSh}{\mathrm{PSh(\mathcal{C})}}
\newcommand{\sPSh}{\mathrm{sPSh(\mathcal{C})}}
\newcommand{\sPShG}{\mathrm{sPSh(\mathcal{C}, G)}}

\newcommand{\Sh}{\mathrm{Sh(\mathcal{C})}}

\newcommand{\Map}{\mathrm{Map}}
\newcommand{\Ob}{\mathrm{Ob}}

\begin{document}
\title{Simplicial presheaves of coalgebras}
\author{George Raptis}
\address{Universit\"{a}t Osnabr\"{u}ck, 
Institut f\"{u}r Mathematik, 
Albrechtstrasse 28a,
49069 Osnabr\"{u}ck, Germany}
\email{graptis@mathematik.uni-osnabrueck.de}
\date{}

\begin{abstract}
The category of simplicial $\mathcal{R}$-coalgebras over a presheaf of commutative unital rings on a small Grothendieck site is endowed with a left proper, simplicial, cofibrantly generated model category structure where the weak equivalences are the local weak equivalences of the underlying simplicial presheaves. This model category is naturally linked to the $\mathcal{R}$-local homotopy theory of simplicial presheaves and the homotopy theory of simplicial $\mathcal{R}$-modules by Quillen adjunctions. We study the comparison with the $\mathcal{R}$-local homotopy category of simplicial presheaves in the special case where $\mathcal{R}$ is a presheaf of algebraically closed (or perfect) fields. If $\mathcal{R}$ is a presheaf of algebraically closed fields, we show that the $\mathcal{R}$-local homotopy category of simplicial presheaves embeds fully faithfully in the homotopy category of simplicial $\mathcal{R}$-coalgebras. 
\end{abstract}

\maketitle

\section{Introduction and statement of results}

Let $\PSh$ denote the category of set-valued presheaves on a small Grothendieck site $\mathcal{C}$. Let $\mathcal{R}$ be a presheaf of commutative unital rings on $\mathcal{C}$. The category $\Mod_{\mathcal{R}}$ of presheaves of $\mathcal{R}$-modules is an abelian locally presentable category with a closed symmetric monoidal pairing that is given by the pointwise tensor product of $\mathcal{R}$-modules. Let $\Coalg_{\mathcal{R}}$ denote the category of cocommutative, coassociative, counital $\mathcal{R}$-coalgebras. The forgetful functor $\Phi: \Coalg_{\mathcal{R}} \to \Mod_{\mathcal{R}}$ is a left adjoint where the right adjoint is the cofree $\mathcal{R}$-coalgebra functor. 

Let $\sPSh$ denote the category of simplicial presheaves on $\mathcal{C}$. The class of \textit{local} weak equivalences between simplicial presheaves defines a homotopy theory that has been studied extensively in the 
literature and forms the subject known as homotopical sheaf theory. Naturally, the subject began with the 
study of the corresponding homotopy theory of simplicial sheaves (see \cite{Br}, \cite{BG}, \cite{Jo}). The shift to the more flexible category of simplicial presheaves is due to Jardine who realized that \textit{``\ldots it is not so much the ambient topos that is creating the homotopy theory as it is the topology of the underlying site''} \cite{Ja1}. On the other hand, the dependence on the choice of a Grothendieck site for the ambient topos is immaterial: a local weak equivalence of simplicial presheves is the same as a local weak equivalence of the associated simplicial sheaves, and so, in particular, the two homotopy theories are equivalent (see \cite{Ja2}). The class of local weak equivalences is a refinement of the sectionwise weak equivalences that takes into account the topology on $\mathcal{C}$. For example, in the case where the topos of sheaves $\Sh$ has enough points, a map of simplicial presheaves is a local weak equivalence if it induces a weak equivalence of simplicial sets at all stalks. Note that this is the same as a sectionwise weak equivalence when $\mathcal{C}$ has the trivial topology. 

Let $\sCoalg_{\mathcal{R}}$ denote the category of simplicial $\mathcal{R}$-coalgebras. The purpose of this paper is to study the homotopy theory of simplicial $\mathcal{R}$-coalgebras and compare it with that of simplicial presheaves. The method to pursue this follows the axiomatic approach of the theory of model categories. We assume that the reader is familiar with the theory of model categories and what it is good for. For background material, we recommend the recent monographs of Hirschhorn \cite{Hi} and Hovey \cite{Ho}. 

A morphism $f: A \to B$ of simplicial $\mathcal{R}$-coalgebras is called a \textit{(local) weak equivalence} if it defines a local weak equivalence between the underlying simplicial presheaves in $\sPSh$. This class of weak equivalences will be denoted here by $\mathcal{W}_{\mathcal{R}}$. A morphism $f: A \to B$ in $\sCoalg_{\mathcal{R}}$ is called a \textit{$\Phi$-monomorphism} if the morphism between the underlying simplicial $\mathcal{R}$-modules is a monomorphism.

\begin{theoremA}
There is a left proper, simplicial, cofibrantly generated model category structure on $\sCoalg _{\mathcal{R}}$ where the class of weak equivalences is $\mathcal{W}_{\mathcal{R}}$ and the set of $\Phi$ - monomorphisms between $\kappa$-presentable objects is a generating set of cofibrations (for any choice of a large enough regular cardinal $\kappa$). 
\end{theoremA}

The choice of $\kappa$ will be clarified in the proof of the theorem in section \ref{TheoremA}. This result was proved by Goerss \cite{G} in the case where $\mathcal{C}$ is the terminal category (i.e., $\sPSh = \sSet$) and $\mathcal{R}$ is given by a single field. Although the structural properties of coalgebras over a field are well-studied in the literature (e.g. see \cite{Swe}), the situation in the general case of a commutative ring is more complicated and less well-understood. Among the basic reasons that make the case of fields special are, of course, the exactness of the tensor product and the available duality techniques. Both of them play a role in Goerss' proof. There is yet another relevant distinctive feature of the case of fields that relates to the (possibility of the) introduction of higher cardinals in the proof of Theorem A. According to the  \textit{fundamental theorem of coalgebras}, due to Sweedler \cite{Swe}, if $\mathbb{F}$ is a field, every $\mathbb{F}$-coalgebra is the filtered colimit of its finite dimensional sub-coalgebras. Moreover, finite dimensional coalgebras are finitely presentable objects in $\Coalg_{\mathbb{F}}$. This property may fail in the case of an arbitrary (presheaf of) commutative ring(s), however an appropriate analogous statement can be formulated in this general case if we allow a higher rank of presentability. This is because the category of $\mathcal{R}$-coalgebras is known to be locally $\lambda$-presentable for some regular cardinal $\lambda$ by results of Barr \cite{Ba} and Fox \cite{F} (see also \cite{Po1}). This fact together with additional background material about $\mathcal{R}$-coalgebras will be discussed in section \ref{background}.

The idea to apply methods from the theory of locally presentable categories in homotopical algebra originates from unpublished work of J. H. Smith. In particular, the theory of combinatorial model categories, introduced by J.H. Smith, has provided the theory of model categories with powerful set-theoretical techniques. Our proof of Theorem A is heavily based on such ideas. We will assume that the reader is familiar with the basic theory of locally presentable and accessible categories. A detailed account can be found in the monograph by Ad\'{a}mek-Rosick\'{y} \cite{AR}. For background material about combinatorial model categories, see the articles \cite{Be}, \cite{Du}, \cite{Ro}.

Following the work of Goerss \cite{G}, the motivation for studying the homotopy theory of simplicial $\mathcal{R}$-coalgebras comes from its connection with the $\mathcal{R}$-\textit{local} homotopy theory of simplicial presheaves, i.e., the Bousfield localization of the model category of simplicial presheaves at the class of $\mathcal{R}$-homology equivalences. Let us first discuss this connection in the case where $\mathcal{C}$ is the terminal category and let $R$ be a commutative ring with unit. The simplicial $R$-chains $R\{X\}$ of a simplicial set $X$ form naturally a simplicial $R$-coalgebra. The comultiplication is induced by the diagonal map $\Delta: X \to X \times X$ and the counit by the map $X \to \Delta^0$. Moreover, there is an adjunction
\begin{equation} \label{adjunction}
\R: \sSet \rightleftarrows \sCoalg_R : \mathcal{\rho}
\end{equation}
where the right adjoint is the functor of $R$-points defined by $$\rho(A)_n=\Coalg_R(R, A_n).$$ An important observation is that the canonical unit map $X \to \mathcal{\rho}R\{X\}$ is an isomorphism, for every simplicial set $X$, when $R$ has 
no non-trivial idempotents. Moreover, the adjunction \eqref{adjunction} is a Quillen adjunction between the standard model category structure on $\sSet$ and the model category of Theorem A. Therefore it induces a new Quillen adjunction 
\begin{equation} \label{adjunction2}
 \R: \mathrm{L}_R \sSet \leftrightarrows \sCoalg_R : \mathcal{\rho}
\end{equation}
where $\mathrm{L}_R \sSet$ denotes the Bousfield localization of $\sSet$ at the $R$-homology equivalences, i.e., the 
class of maps $f: X \to Y$ such that $H_*(f,R)$ is an isomorphism (see \cite{Bo}). Recall that the classical Dold-Kan correspondence (e.g. see \cite{We}) shows that $f$ is an $R$-homology equivalence iff $R\{f\}$ is a weak equivalence of simplicial $R$-modules. 

Then this prompts the question whether the additional coalgebraic structure on the simplicial $R$-chains could suffice in order to produce a faithful appoximation to the homotopy theory of spaces localized at the $R$-homology equivalences. By taking (functorial) cofibrant and fibrant replacements respectively, one obtains a derived adjunction between the respective homotopy categories,
\begin{equation} \label{adjunction3}
\mathbb{L}\R: \mathrm{Ho}(\mathrm{L}_R \sSet) \rightleftarrows \mathrm{Ho}(\sCoalg_R) : \mathbb{R}\mathcal{\rho}
\end{equation}
and the derived unit transformation gives a canonical map $$X \to \mathbb{R}\mathcal{\rho}(R\{X\})$$ from every simplicial set $X$ to an $R$-$local$ space, i.e., a fibrant object in $\mathrm{L}_R \sSet$. Goerss \cite[Theorem C]{G} proved that this map is an $R$-homology equivalence when $R$ is an algebraically closed field. An equivalent statement is that the functor $\mathbb{L} \R$ is fully faithful. More generally, he showed that for any perfect field $\mathbb{F}$ with algebraic closure $\overline{\mathbb{F}}$ and profinite Galois group $G$, the derived unit map can be identified with the map from $X$ into the fixed points of the localization of $X$ at the $\overline{\mathbb{F}}$-homology equivalences where $X$ is regarded as a simplicial $G$-set endowed with the trivial $G$-action (see \cite[Theorem E]{G}). We do not know of any appropriate extension of this remarkable result to an arbitrary ring.

The general case of an arbitrary small site $\mathcal{C}$ and a presheaf of commutative unital rings $\mathcal{R}$ is completely analogous. The (sectionwise) free $\mathcal{R}$-module functor factors through $\sCoalg_{\mathcal{R}}$ 
and there is a Quillen adjunction,
\begin{equation} \label{Q-adjunction1}
 \RR: \sPSh_{\mathrm{inj}} \to \sCoalg_{\mathcal{R}} : \rho
\end{equation}
where $\sPSh_{\mathrm{inj}}$ denotes the model category of simplicial presheaves by Jardine \cite{Ja1}, \cite{Ja2}. Some basic facts about the various model category structures on simplicial presheaves will be reviewed in section \ref{recollect}. The functor of $\mathcal{R}$-points $\rho$ is defined similarly as above, by
\begin{displaymath}
\rho(A)(U)_n =\Coalg_{\mathcal{R}(U)}(\mathcal{R}(U), A(U)_n).
\end{displaymath}
The Quillen adjunction \eqref{Q-adjunction1} will be studied in section \ref{Qadjunctions}. We show that this adjunction can be used to endow $\sCoalg_{\mathcal{R}}$ with a different model category structure where the weak equivalences are pulled back from $\sPSh_{\mathrm{inj}}$ via the functor $\rho$ of $\mathcal{R}$-points. In addition, it follows easily, and somewhat surprisingly, that the adjunction \eqref{Q-adjunction1} defines a Quillen equivalence between this new model category, denoted by $\sCoalg^{\rho}_{\mathcal{R}}$, and $\sPSh_{\mathrm{inj}}$. This produces an alternative point of view for the comparison between the homotopy theory of simplicial presheaves and simplicial $\mathcal{R}$-coalgebras as it can be also modelled by the identity functor $1: \sCoalg^{\rho}_{\mathcal{R}} \to \sCoalg_{\mathcal{R}}$ as a left Quillen functor.

Similarly to the case of a single commutative ring, the comparison via the Quillen adjunction \eqref{Q-adjunction1} can only relate to the $\mathcal{R}$-local part of $\sPSh_{\mathrm{inj}}$. A morphism $f: X \to Y$ between simplicial presheaves is called an \textit{$\mathcal{R}$-homology equivalence} if $\mathcal{R}\{f\}$ is a weak equivalence. The left Bousfield localization $\mathrm{L}_{\mathcal{R}} \sPSh_{\mathrm{inj}}$ of $\sPSh_{\mathrm{inj}}$ at the class of $\mathcal{R}$-homology equivalences exists and so the adjunction \eqref{Q-adjunction1} induces a new Quillen adjunction $\RR: \mathrm{L}_{\mathcal{R}} \sPSh_{\mathrm{inj}} \to \sCoalg_{\mathcal{R}} : \rho$. Based on the structure theory of coalgebras over an algebraically closed field and the methods of Goerss \cite{G}, we prove the following theorem in section \ref{ThmB}.

\begin{theoremB}
Let $\mathscr{F}$ be a presheaf of algebraically closed fields. Then the functor 
\begin{equation*} \label{Q-adjunction2}
\mathbb{L} \mathscr{F}\{-\}: \mathrm{Ho}(\mathrm{L}_{\mathscr{F}} \sPSh_{\mathrm{inj}}) \to \mathrm{Ho}(\sCoalg_{\mathscr{F}})
\end{equation*}
is fully faithful.
\end{theoremB}

Theorem B is a generalization of Goerss' theorem \cite[Theorem C]{G} to the context of simplicial presheaves. We will also consider the case of a constant presheaf at a perfect field, following the case of single perfect field in \cite[Theorem E]{G} as mentioned above. The precise analogue of Theorem B in this case requires some preparatory
work and we will not attempt to summarise it here. It will be discussed in detail in section \ref{ThmB}. \\

\parindent=0in

\textit{Organization of the paper.} In section \ref{background}, we review some categorical properties of the presheaf categories of coalgebras $\Coalg_{\mathcal{R}}$ focusing in particular on the property of local presentability that is crucial to the proof of Theorem A. In section \ref{recollect}, we recall briefly (some of) the various known model categories of simplicial presheaves and simplicial $\mathcal{R}$-modules. In section \ref{TheoremA}, we prove Theorem A. In section \ref{Qadjunctions}, we compare the model category of simplicial $\mathcal{R}$-coalgebras with the model categories of simplicial presheaves and simplicial $\mathcal{R}$-modules. The comparison with simplicial presheaves is based on the adjunction \eqref{Q-adjunction1} as discussed above. The comparison with simplicial $\mathcal{R}$-modules is based on the forgetful functor $\Phi: \sCoalg_{\mathcal{R}} \to \sMod_{\mathcal{R}}$. In section \ref{ThmB}, we prove Theorem B and discuss some generalizations of it to other types of presheaves 
of fields. \\

\parindent=0in

\textit{Acknowledgements.} I would like to thank Manfred Stelzer for the discussions about simplicial coalgebras and his interest in the results of this paper.

\parindent=0.2in

\section{Preliminaries on presheaf categories of Coalgebras} \label{background}

Let $\mathcal{C}$ be a small category and $\PSh:= \mathrm{Fun}(\mathcal{C}^{op}, \mathrm{Set})$ denote the category of set-valued presheaves on $\mathcal{C}$. For every object $U$ of $\mathcal{C}$, there is a presheaf $Y(U): \mathcal{C}^{op} \to \mathrm{Set}$, $V \mapsto \mathcal{C}(V,U)$, called the representable presheaf by $U$. The set of presheaves $\{Y(U) : U \in \Ob \mathcal{C} \}$ defines a strong generator of finitely presentable objects in $\PSh$. This is essentially a consequence of the Yoneda lemma which says that $Y: \mathcal{C} \to \PSh$, $U \mapsto Y(U)$, is fully faithful.

Let $\mathcal{R}$ be a presheaf of commutative unital rings on $\mathcal{C}$. An $\mathcal{R}$-module $M$ is a presheaf of abelian groups on $\mathcal{C}$ such that $M(U)$ is an $\mathcal{R}(U)$-module for every $U \in \Ob \mathcal{C}$ and the restriction map $$M(j): M(U) \to M(V)$$ is a homomorphism of $\mathcal{R}(U)$-modules for every morphism $j: V \to U$ in $\mathcal{C}$. A morphism $f: M \to N$ of $\mathcal{R}$-modules is a natural transformation of the underlying presheaves such that $f|_U: M(U) \to N(U)$ is an $\mathcal{R}(U)$-module homomorphism. This defines a category $\Mod_{\mathcal{R}}$ of $\mathcal{R}$-modules which is abelian and locally ($\aleph_0$-) presentable. We recall that a category is ($\lambda$-)locally presentable if it is cocomplete and has a strong generator of ($\lambda$-)presentable objects \footnote{This is equivalent to the definition of \cite[Definition 1.17]{AR} by \cite[Theorem 1.20]{AR}.}. The forgetful functor $\iota: \Mod_{\mathcal{R}} \to \PSh$ admits a left adjoint $$\RR : \PSh \to \Mod_{\mathcal{R}}$$ that associates to every presheaf $X: \mathcal{C}^{op} \to \mathrm{Set}$, the $\mathcal{R}$-module whose value at $U \in \Ob \mathcal{C}$ is the free $\mathcal{R}(U)$-module with generators $X(U)$. A convenient strong generator for $\Mod_{\mathcal{R}}$ is given by the set of finitely presentable objects $\{\mathcal{R}\{Y(U)\} : U \in \Ob \mathcal{C}\}$.

The category $\Mod_{\mathcal{R}}$ has a closed symmetric monoidal pairing $\otimes : \Mod_{\mathcal{R}} \times \Mod_{\mathcal{R}} \to \Mod_{\mathcal{R}}$ that is given by the sectionwise tensor product of modules. More precisely, this takes a pair $(M,N)$ of $\mathcal{R}$-modules to the $\mathcal{R}$-module $M \otimes N$ whose value at $U \in \Ob \mathcal{C}$ is the $\mathcal{R}(U)$-module $M(U) \otimes_{\mathcal{R}(U)} N(U)$. For every $j: V \to U$ in $\mathcal{C}$, the restriction map is given by the composition of $\mathcal{R}(U)$-module homomorphisms $$M(U) \otimes_{\mathcal{R}(U)} N(U) \to M(V) \otimes_{\mathcal{R}(U)} N(V) \to M(V) \otimes_{\mathcal{R}(V)} N(V).$$
The unit of this monoidal pairing is given by $\mathcal{R}$ as an $\mathcal{R}$-module. 

As for every symmetric monoidal category, there is an associated category of cocommutative, coassociative, counital comonoids with respect to the monoidal pairing. Applied to $(\Mod_{\mathcal{R}}, \otimes, \mathcal{R})$, this defines the category $\Coalg_{\mathcal{R}}$ of cocommutative, coassociative, counital $\mathcal{R}$-coalgebras. More explicitly, an $\mathcal{R}$-coalgebra $(A, \mu, \epsilon)$ is an $\mathcal{R}$-module $A$ together with morphisms of $\mathcal{R}$-modules
for comultiplication 
\begin{displaymath}
 \mu: A \to A \otimes A
\end{displaymath}
and counit 
\begin{displaymath}
 \epsilon : A \to \mathcal{R}
\end{displaymath}
such that the following diagrams commute 
\begin{displaymath}
\xymatrix{
& A \otimes A \ar[dd]^{\stackrel{tw}{\cong}} \\
A \ar[ur]^{\mu} \ar[dr]^{\mu} &  \\
& A \otimes A
}
\end{displaymath}
\begin{displaymath}
\xymatrix{
& A \otimes A \ar[r]^<(0.25){\mu \otimes 1}  & (A \otimes A) \otimes A \ar[dd]^{\cong}\\
A \ar[dr]^{\mu} \ar[ur]^{\mu} & & \\
& A \otimes A \ar[r]^<(0.25){1 \otimes \mu} & A \otimes (A \otimes A)
}
\end{displaymath}
\begin{displaymath}
 \xymatrix{
A \otimes A \ar[d]^{\epsilon \otimes 1} & A \ar[d]^{1} \ar[r]^<(0.3){\mu} \ar[l]_<(0.3){\mu}  & A \otimes A \ar[d]^{1 \otimes \epsilon}\\
\mathcal{R} \otimes A & A \ar[r]^<(0.3){\cong} \ar[l]_<(0.3){\cong} & A \otimes \mathcal{R} 
}
\end{displaymath}

The isomorphisms in these diagrams are the canonical ones, given as part of the symmetric monoidal structure. A morphism of $\mathcal{R}$-coalgebras from $(A, \mu, \epsilon)$ to $(A', \mu', \epsilon')$ is a morphism $f: A \to A'$ between the underlying $\mathcal{R}$-modules that makes the following diagrams commute 
\begin{displaymath}
\xymatrix{
A \otimes A \ar[r]^{f \otimes f} & A' \otimes A' \\
A \ar[u]^{\mu} \ar[r]^f & A' \ar[u]_{\mu'}
}
\end{displaymath}
\begin{displaymath}
\xymatrix{
A \ar[r]^f \ar[d]_{\epsilon} & A' \ar[dl]^{\epsilon'} \\
\mathcal{R} &
}
\end{displaymath}

\parindent=0in

The restriction maps of an $\mathcal{R}$-coalgebra are not quite maps of coalgebras over a ring, since the coalgebraic structure is defined sectionwise with respect to different rings in general. However, note that given a ring homomorphism $f: R \to S$ and an $R$-coalgebra $(A, \mu, \epsilon)$, the $S$-module $A \otimes_R S$ is naturally an $S$-coalgebra with comultiplication defined by
$$ A \otimes_R S \stackrel{\mu \otimes 1}{\longrightarrow} (A \otimes_R A) \otimes_R S \stackrel{\cong}{\to} (A \otimes_R S ) \otimes_S (A \otimes_R S).$$ 
Given an $R$-coalgebra $A$ and an $S$-coalgebra $B$, an $R$-module homomorphism $g: A \to B$ is called a map of coalgebras if the induced $S$-module homomorphism $\tilde{g}: A \otimes_R S \to B$ is a map of $S$-coalgebras. The restriction maps of an $\mathcal{R}$-coalgebra are maps of coalgebras in this sense. 

\parindent=0.2in 

Various important properties of categories of coalgebras have been studied by Barr \cite{Ba} (for a single commutative ring), Fox \cite{F} (for an arbitrary locally presentable monoidal category) and, more recently, 
by Porst \cite{Po1}, \cite{Po2}. We recall some of them.

\begin{proposition} \label{cofreedom}
The forgetful functor $\Phi: \Coalg_{\mathcal{R}} \to \Mod_{\mathcal{R}}$ has a right adjoint. Moreover, $\Coalg_{\mathcal{R}}$ is comonadic over $\Mod_{\mathcal{R}}$.
\end{proposition}
\begin{proof}
This is proved in \cite[Theorem 4.1]{Ba} for the case of a single commutative ring. The same proof applies here, see \cite[Corollary 8]{F} and \cite[Theorem 12, Remark 15]{Po1}.
\end{proof}

The main property that we need in this paper is stated in the following theorem that is essentially due to Barr \cite{Ba} and Fox \cite{F}. First let $\lambda$ be a regular cardinal that is bigger than $\aleph_0$ and the cardinalities of the rings $\mathcal{R}(U)$ for all $U \in \Ob \mathcal{C}$.

\begin{theorem} \emph{(Barr \cite{Ba}-Fox \cite{F})} \label{BarrFox}
$\Coalg_{\mathcal{R}}$ is a locally $\lambda$-presentable category. 
\end{theorem}
\begin{proof}
This is essentially contained in the last remarks of \cite{F} combined with the arguments of \cite[Theorem 3.1, Corollary 3.2]{Ba} to get the required rank of presentability. 

Let us first prove the theorem in the case where $\mathcal{C}$ is a discrete category, i.e., when there are no non-identity morphisms. Consider the set $\mathcal{A}$ of $\mathcal{R}$-coalgebras $A$ whose underlying $\mathcal{R}$-module $\Phi(A)$ satisfies the following properties
\begin{itemize}
\item [(i)] $\Phi(A)$ is the trivial zero module everywhere except for exactly one object $U \in \Ob \mathcal{C}$,
\item[(ii)] the cardinality of $\Phi(A)(U)$ is $\leq$ max$\{\mathrm{card}(\mathcal{R}(U)), \aleph_0\}$. 
\end{itemize}
By \cite[Theorem 3.1]{Ba}, given an $\mathcal{R}$-coalgebra $A$, every element $x \in A(U)$, $U \in \Ob \mathcal{C}$, is contained in an $\mathcal{R}$-subcoalgebra that satisfies (i) and (ii). Similarly to \cite[Corollary 3.2]{Ba}, it follows that $\mathcal{A}$ is a strong generator of $\Coalg_{\mathcal{R}}$. Let $A \in \mathcal{A}$ and $U \in \Ob\mathcal{C}$ as in (i) above. By (ii), $\Phi(A)(U)$ is generated by strictly less than $\lambda$ elements, so it is $\lambda$-presentable as an $\mathcal{R}(U)$-module. By (i), $\Phi(A)$ is also $\lambda$-presentable (in $\Mod_{\mathcal{R}}$). Following the last remarks of \cite{F}, it is immediate that $A$ is also $\lambda$-presentable (in $\Coalg_{\mathcal{R}}$). Indeed it suffices to note that the set of maps from $A$ to an $\mathcal{R}$-coalgebra $B$ can be expressed as the equalizer of a pair of arrows $$\Mod_{\mathcal{R}}(\Phi(A), \Phi(B)) \rightrightarrows \Mod_{\mathcal{R}}(\Phi(A), \Phi(B) \otimes \Phi(B)) \times \Mod_{\mathcal{R}}(\Phi(A), \mathcal{R})$$ and that equalizers commute with $\lambda$-directed colimits (for any $\lambda$) in the category of sets. Hence it follows that $\Coalg_{\mathcal{R}}$ is locally $\lambda$-presentable.

For the general case, let $\mathcal{C}_0$ denote the discrete category with set of objects $\Ob \mathcal{C}$ and $u: \mathcal{C}_0 \to \mathcal{C}$ be the inclusion functor. The presheaf $\mathcal{R}$ clearly restricts to a presheaf of commutative unital rings $\mathcal{R}_0:=u^*\mathcal{R}$ on $\mathcal{C}_0$. There is an adjunction $$u_!: \Coalg_{\mathcal{R}_0} \rightleftarrows \Coalg_{\mathcal{R}}: u^*$$ where the right adjoint is the obvious forgetful pullback functor. The left adjoint is defined by $$u_!(A)(V) =  \bigoplus_{V \to U} A(U) \otimes_{\mathcal{R}(U)} \mathcal{R}(V)$$ where the tensor products are endowed with the natural $\mathcal{R}(V)$-coalgebraic structure as remarked earlier. Moreover, $u^*$ is faithful and preserves ($\lambda$-)directed colimits. It follows that the set $u_!(\mathcal{A})=\{ u_!(A) : A \in \mathcal{A} \}$ ($\mathcal{A}$ as above) is a strong generator of $\lambda$-presentable objects and hence the result also follows. 
\end{proof}

\begin{remark}
The local presentability of $\Coalg_{\mathcal{R}}$ was also proved by Porst \cite{Po1} in the case of a single commutative ring using somewhat different methods based more heavily on general results about accessible categories from \cite{AR}. These methods apply identically to the case of a presheaf of commutative rings, but it is generally difficult to obtain an explicit rank of presentability by them alone and most probably it will be very big. One way to obtain it would be to go through the choices of cardinals in the proofs of \cite[Lemma 2.76, Theorem 2.72, Theorem 2.43, Theorems 2.32-2.34 and Theorem 2.19]{AR} in that order more or less. On the other hand, the rank shown in Theorem \ref{BarrFox} is not the least possible since the category of coalgebras over a field is known to be locally finitely presentable. It would be interesting to know what the best possible rank is exactly and how it is related to the divisibility properties of the elements of $\mathcal{R}$ or the failure of the tensor product to be a left exact functor.
\end{remark}

The tensor product of $\mathcal{R}$-coalgebras $(A, \mu,\epsilon)$ and $(B, \mu',\epsilon')$ is naturally an $\mathcal{R}$-coalgebra. The underlying $\mathcal{R}$-module is $A \otimes B$ and the coalgebraic structure is given by the following $\mathcal{R}$-module morphisms $$A \otimes B \stackrel{\mu \otimes \mu'}{\longrightarrow} (A \otimes A) \otimes (B \otimes B) \cong (A \otimes B) \otimes (A \otimes B)$$ 
$$ A \otimes B \stackrel{\epsilon \otimes \epsilon'}{\longrightarrow} \mathcal{R} \otimes \mathcal{R} \cong \mathcal{R}.$$
There are canonical morphisms
$$p_A: A \otimes B \stackrel{1 \otimes \epsilon'}{\longrightarrow} A \otimes \mathcal{R} \cong A$$
$$p_B: A \otimes B \stackrel{\epsilon \otimes 1}{\longrightarrow}  \mathcal{R} \otimes B \cong B$$
Moreover, the cone $(A \otimes B, p_A, p_B)$ actually defines a product cone for $A$ and $B$ in $\Coalg_{\mathcal{R}}$. Using this description of the products in $\Coalg_{\mathcal{R}}$, the following theorem is an immediate consequence of the special adjoint functor theorem and the fact that $\Phi$ creates colimits.
 
\begin{theorem} \label{cartesian}
$\Coalg_{\mathcal{R}}$ is a cartesian closed category.
\end{theorem}
\begin{proof}
This is proved in \cite[Theorem 5.3]{Ba} for the case of a single commutative ring. The same proof applies here.
\end{proof}

The $\mathcal{R}$-module $\mathcal{R}\{X\}$ has a natural $\mathcal{R}$-coalgebra structure that is induced by the diagonal map $X \stackrel{\Delta}{\to} X \times X$ and the unique map to the terminal object $X \to \ast$. Hence the functor $\RR: \PSh \to \Mod_{\mathcal{R}}$ factors through $\Phi: \Coalg_{\mathcal{R}} \to \Mod_{\mathcal{R}}$. We denote the functor into $\Coalg_{\mathcal{R}}$ again by $\RR: \PSh \to \Coalg_{\mathcal{R}}$. This functor has a right adjoint $\rho: \Coalg_{\mathcal{R}} \to \PSh$ which is defined sectionwise as follows: given an $\mathcal{R}$-coalgebra $A$, then $$\rho(A)(U) = \Coalg_{\mathcal{R}(U)}(\mathcal{R}(U), A(U))$$ defines a presheaf $\rho(A): \mathcal{C}^{op} \to \mathrm{Set}$. The functor $\rho$ preserves $\kappa$-directed colimits, and so $\RR$ preserves $\kappa$-presentable objects, for every regular cardinal $\kappa > \mathrm{card}(Mor \mathcal{C})$. 

The functor $\rho$ is called the functor of $\mathcal{R}$-points because it picks out from every section the elements of the $\mathcal{R}$-coalgebra that co-multiply ``diagonally''. If $\mathcal{R}(U)$ has no non-trivial idempotents for all $U \in \Ob \mathcal{C}$, the set of elements of $\mathcal{R}\{X\}(U)$ that have this property can be identified with $X(U)$.  In this case, the unit transformation of the adjunction $$1 \to \rho \RR$$ is a natural isomorphism.

Note also that since adjoints are essentially unique, the forgetful functor $\iota: \Mod_{\mathcal{R}} \to \PSh$ factors as the composite of the cofree $\mathcal{R}$-coalgebra functor, denoted by $\mathbb{T}: \Mod_{\mathcal{R}} \to \Coalg_{\mathcal{R}}$, followed by the functor of $\mathcal{R}$-points $\rho: \Coalg_{\mathcal{R}} \to \PSh$.

\section{A quick review of the model structures on simplicial presheaves} \label{recollect}

Let $\mathcal{C}$ be a small Grothendieck site and $\sPSh$ denote the category of simplicial presheaves (of sets) on $\mathcal{C}$. The objects are usually understood as diagrams $F: \mathcal{C}^{op} \to \sSet$ and the morphisms are natural transformations of such diagrams. 

Various model category structures on the categories of simplicial presheaves and simplicial $\mathcal{R}$-modules (or presheaves of chain complexes) are known in the literature. We review some facts about the four model category structures that are characterized by the following two specifications: (i) whether the weak equivalences are defined sectionwise or stalkwise, and (ii) whether the cofibrations are defined sectionwise or they are the so-called projective cofibrations.

First we recall the definition of the local weak equivalences in $\sPSh$ from \cite{Ja2}. The definition uses Boolean localization in order to include the case where the associated topos of sheaves does not have enough points. A different definition using sheaves of homotopy groups can be found also in \cite{Ja2}. Let $L^2: \PSh \to \Sh$ denote the sheafification functor. According to a fundamental theorem in topos theory, due to Barr, there is a complete Boolean algebra $\mathcal{B}$ and a surjective geometric morphism $\wp: \mathrm{Sh}(\mathcal{B}) \rightarrow \Sh$ (see \cite[IX.9]{MM} for details). Thus we obtain a geometric morphism between the categories of simplicial objects, $$\wp^*L^2: \sPSh \to \mathrm{sSh}(\mathcal{B}).$$ A map $f: X \to Y$ in $\sPSh$ is a local weak equivalence if (and only if) the map $$(\wp^* L^2 Ex^{\infty}(f)) (b) : (\wp^* L^2 Ex^{\infty} (X))(b) \to (\wp^* L^2Ex^{\infty}(Y))(b)$$ is a weak equivalence of simplicial sets for all $b \in \mathcal{B}$. Here $Ex^{\infty}$ denotes Kan's fibrant replacement functor applied sectionwise. It can be shown that this definition does not depend on the choice of Boolean localization \cite{Ja2}. In particular, a collection of enough points of $\mathrm{Sh}(\mathcal{C})$ (if it exists) defines a Boolean localization, and so a local weak equivalence is a natural transformation that induces a weak equivalence of simplicial sets at every point in this collection. A sectionwise weak equivalence is always a local weak equivalence \cite[Lemma 9]{Ja2}. If $\mathcal{C}$ has the trivial topology, a local weak equivalence is exactly the same as a sectionwise weak equivalence. 

Note that we can dispense with the fibrant replacement functor $Ex^{\infty}$ when the presheaves are already sectionwise fibrant. Thus a morphism $f: M \to N$ between simplicial $\mathcal{R}$-modules is local weak equivalence (of the underlying simplicial presheaves) if (and only if) $\wp^* L^2 (f)(b)$ is a weak equivalence for all $b \in \mathcal{B}$, since a simplicial presheaf of abelian groups is already sectionwise fibrant.

The \textit{projective} model category structure on $\sPSh$ is the standard projective model category structure on a category of diagrams in a cofibrantly generated model category (see \cite{Hi} for details). The weak equivalences (resp. fibrations) are the sectionwise weak equivalences (resp. Kan fibrations) of simplicial sets. In particular, the class of weak equivalences, and so the associated homotopy theory, is independent of the choice of the Grothendieck topology on $\mathcal{C}$. The cofibrations of this model category are characterised by a right lifting property and they will be referred to as the \textit{projective} cofibrations. Furthermore, the projective model category is proper, simplicial and cofibrantly generated. The simplicial structure is defined by the following functorially defined objects,
\begin{eqnarray*}
(K \otimes X)(U) : = K \times X(U) \\
X^K(U) := \Map_{\sSet}(K, X(U)) \\
\Map(X,Y): = \sPSh(\Delta^{\cdot} \otimes X, Y)
\end{eqnarray*}
for every simplicial set $K$ and simplicial presheaves $X$ and $Y$. Sets of generating cofibrations and trivial cofibrations are defined as follows. Let $\mathcal{C}_0$ denote the discrete category with set of objects $\Ob \mathcal{C}$ and $u: \mathcal{C}_0 \to \mathcal{C}$ be the inclusion functor. There is an adjunction $$u_!: \mathrm{sPSh(\mathcal{C}_0)} \rightleftarrows \sPSh: u^*$$ where $u^*$ is the obvious pullback functor. Note that since $u^*$ preserves $\kappa$-directed colimits (in fact, all colimits), the left adjoint $u_!$ preserves $\kappa$-presentable objects (for any $\kappa$). The category $\mathrm{sPSh(\mathcal{C}_0)}$ can be endowed with the product model category structure where all classes of weak equivalences, cofibrations and fibrations are defined pointwise. The projective model category structure on $\sPSh$ is the lifting of this product model category (which is, incidentally, also an example of a projective model category) along the adjunction $u_! \dashv u^*$. A detailed account of the method of transferring a cofibrantly generated model category structure along an adjunction can be found in \cite[Theorem 11.3.2]{Hi}. As generating sets for cofibrations and trivial cofibrations of $\mathrm{sPSh(\mathcal{C}_0)}$ we can choose the ``products'' of some generating sets of the model category 
$\sSet$. Let $I = \{ \partial \Delta^n \hookrightarrow \Delta^n : n \geq 0 \}$ and $J: = \{ \Lambda^n_k \stackrel{\sim}{\hookrightarrow} \Delta^n : 0 \leq k \leq n, n \geq 0 \}$ denote the standard generating sets of $\sSet$. For every $U \in \Ob \mathcal{C}_0$ and simplicial set $K$, let $X(U,K)$ be the presheaf on $\mathcal{C}_0$ which takes the value $K$ at $U$ and $\varnothing$ elsewhere. For every $f: K \to L$ in $\sSet$, there is a natural morphism of presheaves $X(U,f): X(U,K) \to X(U,L)$. The sets $\mathcal{I}_{\mathcal{C}_0} := \{ X(U,f) : U \in \Ob \mathcal{C}_0, f \in I \}$ and  $\mathcal{J}_{\mathcal{C}_0} := \{ X(U,f) : U \in \Ob \mathcal{C}_0, f \in J \}$ are generating sets for cofibrations and trivial cofibrations respectively. Consequently, the sets of morphisms $$\mathcal{I}_{\mathcal{C}} : = u_! ( \mathcal{I}_{\mathcal{C}_0}) $$
$$\mathcal{J}_{\mathcal{C}} := u_! ( \mathcal{J}_{\mathcal{C}_0})$$
are generating sets of cofibrations and trivial cofibrations, respectively, for the projective model category $\sPSh$. Moreover, they consist of finitely presentable objects in $\sPSh^{\to}$.

The \textit{local projective} model category structure has the same cofibrations, i.e., the projective cofibrations, and the weak equivalences are the local weak equivalences. This model category was shown by Blander \cite{Bla}. The notation $\sPSh_{\mathrm{proj}}$ will be used to denote it. It is clearly a left Bousfield localization of the projective model category, so the trivial fibrations are the same in both cases. The local projective model category is again proper, simplicial and cofibrantly generated. Note that the projective model category can be obtained as a special case of the local projective one by endowing the category $\mathcal{C}$ with the trivial topology.

The category $\sMod_{\mathcal{R}}$ inherits a model category structure from $\sPSh_{\mathrm{proj}}$ along the adjunction (cf. \cite[Lemma 2.2]{Ja3}) 
\begin{equation*} \label{mod-proj}
\RR: \sPSh_{\mathrm{proj}} \rightleftarrows \sMod_{\mathcal{R}}: \iota.
\end{equation*}
This can be shown easily using again the standard method of transferring a model category structure along an adjoint pair (see also Remark \ref{proj}). A map $f: M \to N$ in $\sMod_{\mathcal{R}}$ is a weak equivalence (resp. fibration) if $\iota(f)$ is so in $\sPSh_{\mathrm{proj}}$. This local projective model category will be denoted by $\sMod^{\mathrm{proj}}_{\mathcal{R}}$. Note that the trivial fibrations in $\sMod^{\mathrm{proj}}_{\mathcal{R}}$ are exactly the morphisms that define a trivial fibration of the associated presheaves of simplicial sets, so they are the maps that are sectionwise a weak equivalence and a Kan fibration. But a map of simplicial abelian groups is a trivial fibration if and only if it is a weak equivalence and an epimorphism. Thus we obtain the following proposition that will be needed in the proof of Theorem A. 

\begin{proposition} \label{trivialfib}
A map $f: M \to N$ in $\sMod^{\mathrm{proj}}_{\mathcal{R}}$ is a trivial fibration iff it is a sectionwise weak equivalence and an epimorphism. 
\end{proposition}

The model category $\sMod^{\mathrm{proj}}_{\mathcal{R}}$ is proper, simplicial and cofibrantly generated. The simplicial structure on $\sMod_{\mathcal{R}}$ is induced by the simplicial structure of $\sPSh$. More explicitly, given a simplicial set $K$ and simplicial $\mathcal{R}$-modules $M$ and $N$, the simplicial structure is defined by the following objects,
\begin{eqnarray*}
K \otimes M : = \mathcal{R}\{\underline{K}\} \otimes M  \\
M^K(U): = \Map_{\sSet}(K, M(U)) \\
\Map(M,N): = \sMod_{\mathcal{R}}(\Delta^{\cdot} \otimes M, N)
\end{eqnarray*}
where $\underline{K}$ denotes the constant presheaf at the simplicial set $K$. The sets of morphisms $$\mathcal{I}^{\mathrm{proj}}_{\mathcal{R}} : = \mathcal{R} \{ \mathcal{I}_{\mathcal{C}} \}$$ $$\mathcal{J}^{\mathrm{proj}}_{\mathcal{R}} : = \mathcal{R} \{ \mathcal{J}_{\mathcal{C}} \}$$
are generating sets of cofibrations and trivial cofibrations respectively. They also consist of finitely 
presentable objects in $\sMod^{\to}_{\mathcal{R}}$.

There is also an \textit{injective} model category structure on $\sPSh$ due to Heller \cite{He}. The cofibrations and weak equivalences are the sectionwise monomorphisms and sectionwise weak equivalences respectively. This is again independent of the Grothendieck topology on $\mathcal{C}$. In fact, this model category structure is an instance of the more general injective model category structure on a category of diagrams in a combinatorial model category \cite[Proposition A.2.8.2]{Lu}. It is known to be cofibrantly generated, simplicial and proper.

The associated local homotopy theory corresponds to the \textit{local injective} model category structure on $\sPSh$ due to Jardine \cite{Ja1}, \cite{Ja2}. The cofibrations are the monomorphisms and the weak equivalences and the local weak equivalences. As a consequence, it is a left Bousfield localization of the injective model category at the local weak equivalences. It is also cofibrantly generated, simplicial and proper. This model category will be denoted here by $\sPSh_{\mathrm{inj}}$. Again the injective model category is an instance of the local injective one if the category $\mathcal{C}$ is endowed with the trivial topology.

The category of simplicial $\mathcal{R}$-modules inherits also a ``local injective'' model category structure from $\sPSh_{\mathrm{inj}}$ using similar methods as before in the projective case. This is again cofibrantly generated, simplicial and proper \cite{Ja3}. It will be denoted here by $\sMod^{\mathrm{global}}_{\mathcal{R}}$ following Jardine's terminology of \textit{global fibrations}. This model category is different to the local injective model category $\sMod^{\mathrm{inj}}_{\mathcal{R}}$ that we discuss in section \ref{Qadjunctions}. They have the same class of weak equivalences, but the cofibrations of $\sMod^{\mathrm{inj}}_{\mathcal{R}}$ are exactly the monomorphisms. Moreover, $\sMod^{\mathrm{inj}}_{\mathcal{R}}$ is also cofibrantly generated, simplicial and proper (see Theorem \ref{modules}). Let us finally note the directions of the various left Quillen functors,
\begin{displaymath}
\xymatrix{
\sPSh_{\mathrm{proj}} \ar[r]^1 \ar[d]^{\RR} & \sPSh_{\mathrm{inj}} \ar[d]^{\RR} \\
\sMod^{\mathrm{proj}}_{\mathcal{R}} \ar[r]^1 & \sMod^{\mathrm{global}}_{\mathcal{R}} \ar[r]^1 & \sMod^{\mathrm{inj}}_{\mathcal{R}}
}
\end{displaymath}

\section{Proof of Theorem A} \label{TheoremA}

We will apply the following theorem about combinatorial model categories. 

\begin{theorem} \label{j.smith} Let $\mathcal{C}$ be a locally presentable category, $W$ a class of morphisms of $\mathcal{C}$ and $\mathrm{I}$ a set of morphisms. Then the classes of morphisms 
\begin{center}
         $W$, $\mathrm{Cof(I)}$ and $(\mathrm{Cof(I)} \cap W)-\inj$ 
\end{center}
define classes of weak equivalences, cofibrations and fibrations for a cofibrantly generated model category structure on $\mathcal{C}$ if and only if the following conditions are satisfied:
\begin{itemize}
\item[(i)]  $W$ satisfies the 2-out-of-3 property,
\item[(ii)] $\mathrm{I-inj} \subseteq W$,
\item[(iii)] $\mathrm{Cof(I)} \cap W$ is closed under transfinite compositions and pushouts,
\item[(iv)] the (full subcategory spanned by the) class $W$ is accessible and accessibly embedded in $\mathcal{C}^{\rightarrow }$.
\end{itemize}
\end{theorem}
\begin{proof}
Every accessible, accessibly embedded subcategory of a locally presentable category is cone-reflective by \cite[Theorem 2.53]{AR}. Hence it satisfies the solution set condition at every morphism. Moreover, it is closed under retracts \cite[Proposition 1.19]{Be}. Then the sufficiency of the conditions follows from J. H. Smith's recognition theorem \cite[Theorem 1.7]{Be}. The key part of the proof of \cite[Theorem 1.7]{Be} is to use the fact that $W$ is cone-reflective at I in order to obtain a generating set J for $\mathrm{Cof(I)} \cap W$, see \cite[Lemma 1.9]{Be}. The rest of the proof is an easy application of the more standard recognition theorem for cofibrantly generated model categories, see e.g. \cite[Theorem 2.1.19]{Ho}. (This is a small simplification of the proof given in \cite{Be} that we learned from G. Maltsiniotis.). The necessity of (i),(ii) and (iii) is obvious. Proofs of the necessity of (iv) can be found in \cite[Corollary A.2.6.6]{Lu}, \cite[Theorem 4.1]{Ro} and \cite{R}.
\end{proof}

The power of this theorem, when compared to the standard recognition theorem for cofibrantly generated model categories (e.g. see \cite[Theorem 2.1.19]{Ho}), is that it does not assume as given a set of generating 
trivial cofibrations (but it does not produce a very explicit one either), but rather its existence is essentially a consequence of the accessibility properties of the class of weak equivalences. Condition (iv) should be normally the most difficult to verify in the applications of the theorem. The following proposition will be useful.

\begin{proposition} \label{AR250}
Let $F: \mathcal{C} \to \mathcal{D}$ be an accessible functor and $\mathcal{D}'$ be an accessible and accessibly embedded subcategory of $\mathcal{D}$. Then $F^{-1}(\mathcal{D}')$ is an accessible and accessibly embedded subcategory of $\mathcal{C}$.   
\end{proposition}
\begin{proof}
See \cite[Remark 2.50]{AR}.
\end{proof}

We proceed to the proof of Theorem A with the verification of the conditions (i)-(iv). Condition (i) is obviously satisfied by the class $\mathcal{W}_{\mathcal{R}}$.

\begin{proposition} \label{accessibility}
The class of weak equivalences $\mathcal{W}_{\mathcal{R}}$ is accessible and accessibly embedded in $\sCoalg_{\mathcal{R}}^{\to}$.
\end{proposition}
\begin{proof}
By Theorem \ref{j.smith}, the class of local weak equivalences in $\sPSh$ is accessible and accessibly embedded in $\sPSh^{\to}$. The forgetful functor $\sCoalg_{\mathcal{R}} \to \sPSh$ is accessible being the composition of the forgetful left adjoint $\Phi: \sCoalg_{\mathcal{R}} \to \sMod_{\mathcal{R}}$ (which preserves all colimits) followed by the forgetful right adjoint $\iota: \sMod_{\mathcal{R}} \to \sPSh$ (which preserves directed colimits). It follows that the class $\mathcal{W}_{\mathcal{R}}$ is accessible and accessibly embedded in $\sCoalg_{\mathcal{R}}^{\to}$ by Proposition \ref{AR250}.  
\end{proof}

Let $\kappa$ be a regular cardinal such that $\sCoalg_{\mathcal{R}}$ is locally $\kappa$-presentable. For example, this can be the choice of cardinal from Theorem \ref{BarrFox}. Also assume that $\kappa > \mathrm{max}\{\mathrm{card}(Mor\mathcal{C}), \aleph_0 \}$. 

Let $\mathcal{I}$ denote the set of $\Phi$-monomorphisms between $\kappa$-presentable objects in $\sCoalg_{\mathcal{R}}$. Recall that a map $f: A \to B$ in $\sCoalg_{\mathcal{R}}$ is called a $\Phi$-monomorphism if the map between the underlying simplicial $\mathcal{R}$-modules is a monomorphism. It is clear that every morphism in $\mathrm{Cof(\mathcal{I})}$ is a $\Phi$-monomorphism. So it follows that the class $\mathrm{Cof(\mathcal{I})} \cap \mathcal{W}_{\mathcal{R}}$ is closed under pushouts and transfinite compositions since they are created in $\sMod^{\mathrm{inj}}_{\mathcal{R}}$. So it remains to verify condition (ii) of Theorem \ref{j.smith}. First we prove the following key lemma.

\begin{lemma} \label{factorisation}
Every map $f: A \to B$ in $\sCoalg_{\mathcal{R}}$ admits a factorization $f=pi$ in $\sCoalg_{\mathcal{R}}$ 
such that the following are satisfied:
\begin{itemize}
\item [(a)] $i$ is a $\Phi$-monomorphism,
\item[(b)]  the domain of $p$ is $\kappa$-presentable if both $A$ and $B$ are $\kappa$-presentable,
\item[(c)] $\Phi(p)$ is a trivial fibration in $\sMod^{\mathrm{proj}}_{\mathcal{R}}$.
\end{itemize}
\end{lemma}
\begin{proof} 
Let $\underline{\mathcal{R}}$ denote the constant simplicial $\mathcal{R}$-coalgebra at $\mathcal{R}$ viewed as an $\mathcal{R}$-coalgebra. This is the same as $\mathcal{R}\{\underline{\Delta}^0\}$. (Recall that $\underline{\Delta}^n$ denotes the constant simplicial presheaf whose value is the standard $n$-simplex everywhere.). The fold map $\underline{\mathcal{R}} \oplus \underline{\mathcal{R}} \to \underline{\mathcal{R}}$ in $\sCoalg_{\mathcal{R}}$ admits a factorization as required, induced by the factorization in $\sPSh$,  
\begin{displaymath}
\underline{\Delta}^0 \sqcup \underline{\Delta}^0 \stackrel{i_0 \sqcup i_1}{\longrightarrow} \underline{\Delta}^1 \stackrel{p}{\to} \underline{\Delta}^0
\end{displaymath}
and applying the functor $\RR: \sPSh \to \sCoalg_{\mathcal{R}}$ (This functor will be discussed in detail in section \ref{comparisonI}). In the general case of a map $f: A \to B$, define the mapping cylinder $M(f)$ in the standard way by a pushout diagram in $\sCoalg_{\mathcal{R}}$,
\begin{displaymath}
\xymatrix{
A \cong A \otimes \underline{\mathcal{R}} \ar[r]^{1 \otimes \mathcal{R}\{i_0\} } \ar[d]^f & A \otimes \mathcal{R}\{\underline{\Delta}^1\} \ar[d] \\
B \ar[r]^{j} & M(f) 
}
\end{displaymath}
The mapping cylinder construction yields a factorization of $f: A \to B$ as
\begin{displaymath}
A  \stackrel{i}{\rightarrow} M(f) \stackrel{p}{\to} B
\end{displaymath}
in the usual way.  The map $\Phi(i)$ is clearly a monomorphism of simplicial presheaves, so (a) is satisfied. If $A$ is $\kappa$-presentable then so is $A \otimes \mathcal{R}\{\underline{\Delta}^1\}$, and therefore (b) is also satisfied. Note that $p$ is a split epimorphism since $pj = 1_B$. The map $$\Phi(1 \otimes \mathcal{R}\{i_0\}): \Phi(A) \to \Phi(A \otimes \mathcal{R}\{\underline{\Delta}^1\}) \cong \Phi(A) \otimes \mathcal{R}\{\underline{\Delta}^1\}$$ is a sectionwise weak equivalence and a monomorphism. Since the pushout square above also defines a pushout of simplicial $\mathcal{R}$-modules, it follows that $\Phi(j)$ is a sectionwise weak equivalence. Hence $\Phi(p)$ is a sectionwise epimorphism and weak equivalence, so by Proposition \ref{trivialfib}, it is a trivial fibration in $\sMod^{\mathrm{proj}}_{\mathcal{R}}$ as required by (c).
\end{proof}

\begin{proposition} \label{orthogonality}
$\mathcal{I} - \mathrm{inj} \subseteq \mathcal{W}_{\mathcal{R}}$.
\end{proposition}
\begin{proof}
Let $q: X \to Y$ be a map in $\mathcal{I}- \inj$. It suffices to show that there is lift to every diagram in $\sMod_{\mathcal{R}}$
\begin{displaymath}
\xymatrix{
\cdot \ar[d]_t \ar[r] & \Phi(X) \ar[d]^{\Phi(q)} \\
\cdot \ar[r] & \Phi(Y) 
}
\end{displaymath}
where $t$ is chosen from the generating set for cofibrations of the model category $\sMod^{\mathrm{proj}}_{\mathcal{R}}$ as discussed in section \ref{recollect}. By the choice of the cardinal $\kappa$, the map $q$ is a $\kappa$-directed colimit of $\kappa$-presentable objects in $\sCoalg_{\mathcal{R}}^{\to}$. Since the morphism $t$ is finitely presentable in $\sMod_{\mathcal{R}}^{\to}$ and $\Phi$ preserves colimits, there is a factorization 
\begin{displaymath}
\xymatrix{
\cdot \ar[r] \ar[d]_t & \Phi(A) \ar[d]^{\Phi(f)} \ar[r]^{\Phi(\alpha)} & \Phi(X) \ar[d]^{\Phi(q)} \\
\cdot \ar[r] & \Phi(B) \ar[r]^{\Phi(\beta)} & \Phi(Y) 
}
\end{displaymath}
where $A$ and $B$ are $\kappa$-presentable and the right-hand side square is the image of a commutative square in $\sCoalg_{\mathcal{R}}$ under $\Phi$. If we factorize the map $f$ as in Lemma \ref{factorisation}, we 
obtain a commutative diagram
\begin{displaymath}
\xymatrix{
& \Phi(A) \ar[d]^{\Phi(i)} \ar[r]^{\Phi(\alpha)} & \Phi(X) \ar[dd]^{\Phi(q)} \\
\cdot \ar[ur] \ar[r] \ar[d]_t & \Phi(C) \ar@{.>}[ur]_{\Phi(h_2)} \ar[d]^{\Phi(p)} & \\
\cdot \ar@{.>}[ur]_{h_1} \ar[r] & \Phi(B) \ar[r]^{\Phi(\beta)} & \Phi(Y)
}
\end{displaymath}
By the assumption that $\Phi(p)$ is a trivial fibration in $\sMod^{\mathrm{proj}}_{\mathcal{R}}$, it follows that there exists a lift $h_1$ as indicated in the diagram. By the construction of the factorization in Lemma \ref{factorisation}, $C$ is $\kappa$-presentable, so the map $i$ is in $\mathcal{I}$. Therefore there is a morphism $h_2: C \to X$ such that the following diagram in $\sCoalg_{\mathcal{R}}$ commutes
\begin{displaymath}
\xymatrix{
A \ar[d]_i \ar[r]^{\alpha} & X \ar[d]^q \\
C \ar[ur]^{h_2} \ar[r]^{\beta p} & Y 
}
\end{displaymath}
Then composition $h= \Phi(h_2) h_1$ provides a lift to the original diagram and hence the result follows.
\end{proof}

By Theorem \ref{j.smith}, it follows that the classes of weak equivalences $\mathcal{W}_{\mathcal{R}}$ and cofibrations $\mathrm{Cof(\mathcal{I})}$ define a cofibrantly generated model category structure on $\sCoalg_{\mathcal{R}}$. It is left proper because $\sMod^{\inj}_{\mathcal{R}}$ is left proper (see Theorem \ref{modules}) and the forgetful functor $$\Phi: \sCoalg_{\mathcal{R}} \to \sMod^{\mathrm{inj}}_{\mathcal{R}}$$ preserves pushouts.

The simplicial structure on $\sCoalg_{\mathcal{R}}$ is defined as follows. For every simplicial set $K$ and simplicial $\mathcal{R}$-coalgebra $A$, the tensor structure is induced by $\sMod_{\mathcal{R}}$, i.e.,
\begin{equation} \label{tensor}
K \otimes A : = \mathcal{R}\{\underline{K} \} \otimes A.
\end{equation}
More explicitly, the $n$-simplices of $(K \otimes A)(U)$, $U \in \Ob \mathcal{C}$, is the tensor product $\mathcal{R}(U)$-coalgebra 
$$\mathcal{R}(U)\{K_n\} \otimes_{\mathcal{R}(U)} A(U)_n$$ 
where the coalgebraic structure on the free $\mathcal{R}(U)$-module $\mathcal{R}(U)\{ K_n \}$ is induced by the canonical maps $\Delta: K_n \to K_n \times K_n$ and $K_n \to \ast$. This defines a functor $\sSet \times \sCoalg_{\mathcal{R}} \to \sCoalg_{\mathcal{R}}$ which preserves colimits in both variables. Hence it extends to an adjunction of two variables in the sense of \cite[Definition 4.1.12]{Ho} by the special adjoint functor theorem. Then it suffices to show that given a monomorphism $i: K \hookrightarrow L$ between finitely presentable simplicial sets and a $\Phi$-monomorphism $f: A \to B$ in $\mathcal{I}$, then the morphism of the pushout product $$i \square f:  K \otimes B \cup_{K \otimes A} L \otimes A \to L \otimes B$$
is a $\Phi$-monomorphism and it is trivial if either $i$ or $f$ is trivial. Both the domain and codomain of $i \square f$ are again $\kappa$-presentable. Since $\Phi$ preserves pushouts, the morphism $\Phi(i \square f)$ is isomorphic to $$i \square \Phi(f): K \otimes \Phi(B) \cup_{K \otimes \Phi(A)} L \otimes \Phi(A) \to L \otimes \Phi(B).$$
This is a monomorphism and it is trivial if either $i$ or $\Phi(f)$ is trivial because $\sMod^{\mathrm{inj}}_{\mathcal{R}}$ is a simplicial model category by Theorem \ref{modules}. Hence the simplicial structure of \eqref{tensor} makes the model category $\sCoalg_{\mathcal{R}}$ into a simplicial model category.

This concludes the proof of Theorem A.

\begin{remark}
It is not clear whether there is a model category structure such that the cofibrations are all the $\Phi$-monomorphisms. The standard argument to show that this class is cofibrantly generated (e.g. see \cite[Proposition 1.12]{Be}) does not apply here since $\sCoalg_{\mathcal{R}}$ is not closed under the 
intersection of subobjects in $\sMod_{\mathcal{R}}$ because $\otimes: \Mod_{\mathcal{R}} \times \Mod_{\mathcal{R}} \to \Mod_{\mathcal{R}}$ is not left exact in general. For the same reason, the tensor product of simplicial $\mathcal{R}$-coalgebras, which by Theorem \ref{cartesian} gives the product functor in $\sCoalg_{\mathcal{R}}$, 
does not define a monoidal model category in general. 
\end{remark}

\section{Comparison with simplicial presheaves and simplicial $\mathcal{R}$-modules} \label{Qadjunctions}

\subsection{}\label{comparisonI} The homotopy theory of simplicial presheaves and simplicial $\mathcal{R}$-coalgebras are linked by the functor of simplicial $\mathcal{R}$-chains $\RR: \sPSh \to \sCoalg_{\mathcal{R}}$. This takes a simplicial presheaf $X: \mathcal{C}^{op} \to \sSet$ to the simplicial $\mathcal{R}$-coalgebra $\mathcal{R}\{X\}$ whose underlying $\mathcal{R}$-module is the free $\mathcal{R}$-module on $X$ (denoted also by $\mathcal{R}\{X\}$) and the coalgebraic structure is induced by the canonical maps 
$$X \stackrel{\Delta}{\rightarrow} X \times X$$ 
$$X \rightarrow \underline{\Delta}^0.$$ 
More explicitly, $\mathcal{R}\{X\}: \Delta^{op} \to \Coalg_{\mathcal{R}}$ is defined pointwise, by $$\mathcal{R}\{X\}_{n}(U) = \mathcal{R}(U)\{X(U)_n\}$$
with the coalgebraic structure induced similarly pointwise. This functor has a right adjoint $\rho: \sCoalg_{\mathcal{R}} \to \sPSh_{\mathrm{inj}}$ which is defined pointwise as follows: given $A: \Delta^{op} \to \Coalg_{\mathcal{R}}$, then $$\rho(A)_n(U) = \Coalg_{\mathcal{R}(U)}(\mathcal{R}(U), A_n(U))$$ defines a simplicial presheaf $\rho(A): \Delta^{op} \to \PSh$. This adjunction is induced by the analogous adjunction $\RR: \PSh \rightleftarrows \Coalg_{\mathcal{R}}: \rho$ from section \ref{background}.

\begin{proposition} \label{prop1}
The adjunction $\RR: \sPSh_{\mathrm{inj}} \rightleftarrows \sCoalg_{\mathcal{R}}: \rho$ is a Quillen adjunction.
\end{proposition}
\begin{proof}
We check that $\RR$ preserves cofibrations and trivial cofibrations. A generating set $\mathcal{I}_{\mathrm{inj}}$ for the class of monomorphisms in $\sPSh$ is given by all monomorphisms between $\kappa$-presentable objects. This is a consequence of the general statement of \cite[Proposition 1.12]{Be} combined with some basic properties of the rank of presentability of presheaves, see e.g. \cite[Example 1.31]{AR}. The image of a $\kappa$-presentable object under $\RR$ is again $\kappa$-presentable. Thus the monomorphisms in $\mathcal{I}_{\mathrm{inj}}$ maps to (generating) cofibrations in $\sCoalg_{\mathcal{R}}$. It follows that $\RR$ preserves cofibrations. It also preserves trivial cofibrations because it preserves all weak equivalences (cf. \cite[Lemma 2.1]{Ja3}).
\end{proof}

There is a refinement of the Quillen adjunction above that offers a more precise comparison. This is obtained by localizing the category of simplicial presheaves at the class of $\mathcal{R}$-homology equivalences, i.e., the morphisms $f: X \to Y$ such that $\mathcal{R}\{f\}$ is a weak equivalence. Note that every local weak equivalence is an $\mathcal{R}$-homology equivalence (e.g. see \cite[Lemma 2.1]{Ja3}). The class of $\mathcal{R}$-homology equivalences is the class of weak equivalences for a new model category structure on $\sPSh$ which can be obtained as a left Bousfield localization of $\sPSh_{\mathrm{inj}}$. For background material about the Bousfield localization of 
model categories, see Hirschhorn \cite{Hi}.

\begin{theorem} \label{Bousfield}
The left Bousfield localization $\mathrm{L}_{\mathcal{R}} \sPSh_{\mathrm{inj}}$ of the model category $\sPSh_{\mathrm{inj}}$ at the class of $\mathcal{R}$-homology equivalences exists, and $$\RR: \mathrm{L}_{\mathcal{R}}\sPSh_{\mathrm{inj}} \rightleftarrows \sCoalg_{\mathcal{R}}: \rho$$ is a 
Quillen adjunction.
\end{theorem}
\begin{proof}
The class of $\mathcal{R}$-homology equivalences is the inverse image of $\mathcal{W}_{\mathcal{R}}$, which is accessible and accessibly embedded in $\sCoalg_{\mathcal{R}}^{\to}$ by Proposition \ref{accessibility}, by the accessible functor $\RR^{\to}: \sPSh^{\to} \to \sCoalg_{\mathcal{R}}^{\to}$. Therefore it is is accessible and accessibly embedded in $\sPSh^{\to}$ by Proposition \ref{AR250}. The existence of the Bousfield localization follows from Theorem \ref{j.smith}: conditions (i), (ii) and (iv) are satisfied and (iii) is an easy consequence of the corresponding condition for $\sCoalg_{\mathcal{R}}$ and Proposition \ref{prop1}. Then it is clear that $\RR: \mathrm{L}_{\mathcal{R}}\sPSh_{\mathrm{inj}} \to \sCoalg_{\mathcal{R}}$ is a left Quillen functor. 
\end{proof}

As a consequence, there is a derived adjunction $$\mathbb{L}\RR: \mathrm{Ho}(\mathrm{L}_{\mathcal{R}}\sPSh_{\mathrm{inj}}) \rightleftarrows \mathrm{Ho}(\sCoalg_{\mathcal{R}}): \mathbb{R}\rho$$  between the $\mathcal{R}$-local homotopy category of simplicial presheaves and the homotopy category of simplicial $\mathcal{R}$-coalgebras.

\subsection{} Assume that $\mathcal{R}(U)$ has no non-trivial idempotents for all $U \in \Ob \mathcal{C}$. In this case, the unit transformation $1 \to \rho \RR$ is a natural isomorphism. The adjunction $(\RR, \rho)$ can be used to produce a new model category structure on $\sCoalg_{\mathcal{R}}$. We will need the following elementary lemma. 

\begin{lemma} \label{key2}
Let $i: X \hookrightarrow Y$ be a monomorphism in $\sPSh$. If 
\begin{equation} \label{pushout-lemma}
\xymatrix{
\mathcal{R}\{X\} \ar[r] \ar[d]^{\mathcal{R}\{i\}} & A \ar[d]^{j} \\
\mathcal{R}\{Y\} \ar[r] & C 
}
\end{equation}
is a pushout square in $\sCoalg_{\mathcal{R}}$, then the adjoint square 
\begin{displaymath}
\xymatrix{
X \ar[r] \ar[d]^{i} & \rho(A) \ar[d]^{\rho(j)} \\
Y \ar[r] & \rho(C) 
}
\end{displaymath}
is a pushout in $\sPSh$.
\end{lemma}
\begin{proof}
Since pushouts are computed pointwise, it suffices to check this in the case of a pushout diagram \eqref{pushout-lemma} in $\Coalg_R$ where $R$ is a single commutative ring and $R\{-\}: \mathrm{Set} \to \Coalg_{R}$. In this case, the $R$-coalgebra $C$ is isomorphic to the direct sum of the $R$-coalgebra $A$ with $\oplus_{\alpha \in Y - X} R_{\alpha}$ where each $R_{\alpha}$ is isomorphic to $R$ regarded as an $R$-coalgebra. Then it is easy to check that the $R$-points of $C$ is the disjoint union of the $R$-points of $A$ with the set $Y - X$. 
\end{proof}

A morphism $f: A \to B$ in $\sCoalg_{\mathcal{R}}$ is called a $\rho$-weak equivalence (resp. $\rho$-fibration) if the map $\rho(f)$ is a local weak equivalence (resp. global fibration, i.e., a fibration in the model category $\sPSh_{\mathrm{inj}}$). Let $\mathcal{W}^{\rho}_{\mathcal{R}}$ and $\mathrm{Fib}^{\rho}_{\mathcal{R}}$ denote the classes of $\rho$-weak equivalences and $\rho$-fibrations respectively, and let $\mathrm{Cof}^{\rho}_{\mathcal{R}}$ denote the class of $\rho$-cofibrations, that is, morphisms that have the left lifting property with respect to all maps that are both $\rho$-weak equivalences and $\rho$-fibrations.

\begin{theorem} \label{rho}
There is a proper, simplicial, cofibrantly generated model category $\sCoalg^{\rho}_{\mathcal{R}}$ whose underlying category is $\sCoalg_{\mathcal{R}}$ and the weak equivalences, fibrations and cofibrations are defined by the classes $\mathcal{W}^{\rho}_{\mathcal{R}}$, $\mathrm{Fib}^{\rho}_{\mathcal{R}}$ and $\mathrm{Cof}^{\rho}_{\mathcal{R}}$ respectively.  Moreover, the adjunction $$\RR: \sPSh_{\mathrm{inj}} \rightleftarrows \sCoalg^{\rho}_{\mathcal{R}}: \rho$$ is a Quillen equivalence.
\end{theorem}
\begin{proof} Using the standard method of transferring a model category structure along an adjunction \cite[Theorem 11.3.2]{Hi}, it suffices to check that for every pushout diagram in $\sCoalg_{\mathcal{R}}$
\begin{displaymath}
\xymatrix{
\mathcal{R}\{X\} \ar[r] \ar[d]^{\mathcal{R}\{i\}} & A \ar[d]^{j} \\
\mathcal{R}\{Y\} \ar[r] & C
}
\end{displaymath}
where $i:X \stackrel{\sim}{\hookrightarrow} Y$ is a trivial cofibration, then the morphism $j$ is $\rho$-weak equivalence. But this follows directly by Lemma \ref{key2}. For generating sets of cofibrations and trivial cofibrations, we can choose $\mathcal{R}\{\mathcal{I}_{\mathrm{inj}}\}$ and $\mathcal{R}\{\mathcal{J}_{\mathrm{inj}}\}$ respectively, where $\mathcal{I}_{\mathrm{inj}}$ and $\mathcal{J}_{\mathrm{inj}}$ denote generating sets of $\sPSh_{\mathrm{inj}}$.

We show that $\sCoalg^{\rho}_{\mathcal{R}}$ is a simplicial model category. Let $I= \{ \partial \Delta^n \hookrightarrow \Delta^n | n \geq 0 \}$ and $J = \{ \Lambda^n_k \hookrightarrow \Delta^n | 0 \leq k \leq n \}$ be the standard generating sets of cofibrations and trivial cofibrations of $\sSet$. The simplicial structure is the same as that of Theorem A, i.e., it is defined by 
$$\otimes : \sSet \times \sCoalg^{\rho}_{\mathcal{R}} \to \sCoalg^{\rho}_{\mathcal{R}}$$ 
$$K \otimes A = \mathcal{R}\{\underline{K}\} \otimes A.$$ 
By \cite[Corollary 4.2.5]{Ho}, it suffices to show that the pushout products in $I \Box \mathcal{R}\{\mathcal{I}_{\mathrm{inj}}\}$ are $\rho$-cofibrations and those in $J \Box \mathcal{R}\{\mathcal{I}_{\mathrm{inj}}\}$ and $I \square \mathcal{R}\{\mathcal{J}_{\mathrm{inj}}\}$ are $\rho$-weak equivalences. The pushout product of $i: K \to L$ with $f: A \to B$ is the canonical morphism $$i \square f : K \otimes B \cup_{K \otimes A} L \otimes A \to L \otimes B.$$
The pushout product of $i: K \hookrightarrow L$ in $\sSet$ with $\mathcal{R}\{j\}: \mathcal{R}\{X \} \hookrightarrow \mathcal{R}\{Y \}$ is $\mathcal{R}\{i \square j \}$, where $i \square j$ denotes the pushout 
product of $i$ and $j$ with respect to the simplicial structure of $\sPSh$. Thus the required result follows from the fact that $\sPSh_{\mathrm{inj}}$ is a simplicial model category. 

The model category $\sCoalg^{\rho}_{\mathcal{R}}$ right proper because $\sPSh_{\mathrm{inj}}$ is right proper and $\rho$ preserves pullbacks. Left properness follows easily from Lemma \ref{key2} and the fact that $\sPSh_{\mathrm{inj}}$ is left proper. 

Lastly we show that $\RR$ is a left Quillen equivalence. It suffices to check that the derived unit transformation 
is a natural isomorphism. This holds because $\RR \dashv \rho$ is a coreflection and $\rho$ preserves the weak equivalences by definition. More explicitly, the derived unit transformation of the Quillen adjunction is defined as follows: given an object $X$ of $\sPSh_{\mathrm{inj}}$, let $$g: \mathcal{R}\{X\} \stackrel{\sim}{\to} X^f$$ be a functorial fibrant replacement in $\sCoalg^{\rho}_{\mathcal{R}}$ obtained by an application of the small-object argument to the set of trivial cofibrations $\mathcal{R}\{\mathcal{J}_{\mathrm{inj}}\}$. Then the derived unit trasformation at $X$ can be represented by the map $$X \cong \rho \mathcal{R}\{X\} \stackrel{\rho(g)}{\longrightarrow} \rho(X^f).$$ By Lemma \ref{key2}, this is in $\mathcal{J}_{\mathrm{inj}}$-cell, so in particular it is a local weak equivalence. 
\end{proof}

\begin{proposition}
The identity functor $1: \sCoalg^{\rho}_{\mathcal{R}} \to \sCoalg_{\mathcal{R}}$ is a left Quillen functor between the model categories of Theorem \ref{rho} and Theorem A.
\end{proposition}
\begin{proof}
Let $\mathcal{I}_{\mathrm{inj}}$ be the generating set of monomorphisms in $\sPSh$ that consists of the monomorphisms between $\kappa$-presentable objects and $\mathcal{J}_{\mathrm{inj}}$ a generating set of trivial cofibrations. Then $\mathcal{R}\{\mathcal{I}_{\mathrm{inj}}\}$ and $\mathcal{R}\{\mathcal{J}_{\mathrm{inj}}\}$ are generating sets of cofibrations and trivial cofibrations for $\sCoalg^{\rho}_{\mathcal{R}}$. Every morphism in $\mathcal{R}\{\mathcal{I}_{\mathrm{inj}}\}$ is a $\Phi$-monomorphism between  $\kappa$-presentable objects and therefore the identity functor preserves cofibrations. Every morphism in $\mathcal{R}\{\mathcal{J}_{\mathrm{inj}}\}$ is a local weak equivalence. Hence it follows that $1: \sCoalg^{\rho}_{\mathcal{R}} \to \sCoalg_{\mathcal{R}}$ is a left Quillen functor.
\end{proof}

\begin{remark}
By Lemma \ref{key2}, it is easy to see that the cofibrant objects in $\sCoalg^{\rho}_{\mathcal{R}}$ are exactly the objects of the form $\mathcal{R}\{X\}$ for some simplicial presheaf $X$. Thus a cofibrant replacement functor in $\sCoalg^{\rho}_{\mathcal{R}}$ is given by the counit transformation of the adjunction $(\RR, \rho)$, i.e., the natural morphism $\mathcal{R}\{\rho(A)\} \to A$ is a cofibrant replacement of the simplicial $\mathcal{R}$-coalgebra $A$.
\end{remark}

\subsection{} By Proposition \ref{cofreedom}, the forgetful functor $\Phi: \sCoalg_{\mathcal{R}} \to \sMod_{\mathcal{R}}$ is a left adjoint. It is a left Quillen functor as long as there is a model category 
structure on $\sMod_{\mathcal{R}}$ where the weak equivalences are the local weak equivalences of the inderlying simplicial presheaves and there are enough cofibrations. The following theorem is undoubtly well-known to the experts
but we were not able to find an exact reference for it in the literature. 

\begin{theorem} \label{modules}
There is a proper, simplicial, cofibrantly generated model category structure on $\sMod_{\mathcal{R}}$ where the cofibrations are the monomorphisms and the weak equivalences are the local weak equivalences of the underlying simplicial presheaves. 
\end{theorem}
\begin{proof}
The proof will follow the method of Theorem \ref{j.smith}.  Let us denote again by $\mathcal{W}_{\mathcal{R}}$ the class of local weak equivalences in $\sMod_{\mathcal{R}}$. The class of local weak equivalences in $\sPSh$ is accessible and accessibly embedded in $\sPSh_{\mathrm{inj}}^{\to}$ by Theorem \ref{j.smith}. The forgetful functor $\iota: \sMod_{\mathcal{R}} \to \sPSh$ is accessible, therefore condition (iv) of Theorem \ref{j.smith} holds by Proposition \ref{AR250}. The class of monomorphisms $\mathrm{Mono}$ in $\sMod_{\mathcal{R}}^{\to}$ is cofibrantly generated by a set of monomorphisms. This is more generally true in every Grothendieck abelian category, see \cite[Proposition 1.12, Remark 1.13]{Be}. Next we show that the class $\mathrm{Mono} \cap \mathcal{W}_{\mathcal{R}}$ is closed under pushouts and transfinite compositions. The closure under transfinite compositions is obvious (since those can be computed in $\sPSh_{\mathrm{inj}}$), so it suffices to show that for every pushout square 
\begin{displaymath}
\xymatrix{
A \ar[d]^j \ar[r] & X \ar[d]^f \\
B \ar[r] & Y
}  
\end{displaymath}
where $j \in \mathrm{Mono} \cap \mathcal{W}_{\mathcal{R}}$, then $f \in \mathcal{W}_{\mathcal{R}}$. Let $\wp: \mathrm{Sh}(\mathcal{B}) \rightarrow \Sh$ be a Boolean localization of $\Sh$. Recall the definition of the 
local weak equivalences from section \ref{recollect} and the terminology used there. There is
a pushout diagram
\begin{displaymath}
\xymatrix{
\wp^* L^2(A)(b) \ar[d]^{j(b)} \ar[r] & \wp L^2(X)(b) \ar[d]^{f(b)} \\
\wp^*L^2(B)(b) \ar[r] & \wp^*L^2(Y)(b)
}  
\end{displaymath}
for all $b \in \mathcal{B}$. Since $\wp^*$ and $L^2$ are geometric morphisms, they preserve monomorphisms, so $j(b)$ is a monomorphism. It is also a weak equivalence of simplicial sets by assumption. Then it follows that $f(b)$ is also a weak equivalence and so condition (iii) of Theorem \ref{j.smith} follows. 

It remains to verify condition (ii) of Theorem \ref{j.smith}. Let $\mathcal{I}^{\mathrm{proj}}_{\mathcal{R}} = \mathcal{R}\{\mathcal{I}_{\mathcal{C}}\}$ be the generating set of cofibrations for $\sMod^{\mathrm{proj}}_{\mathcal{R}}$ as defined in section \ref{recollect}. Clearly $\mathcal{I}^{\mathrm{proj}}_{\mathcal{R}} \subseteq \mathrm{Mono}$, so $\mathrm{Mono-inj} \subseteq \mathcal{I}^{\mathrm{proj}}_{\mathcal{R}}$-inj. But if $f \in \mathcal{I}^{\mathrm{proj}}_{\mathcal{R}}$-inj then  $f$ is a local weak equivalence, so (ii) follows. Hence by Theorem \ref{j.smith}, there is a cofibrantly generated model category, denoted by $\sMod^{\mathrm{inj}}_{\mathcal{R}}$, as required. 

This model category is right proper because $\sPSh_{\mathrm{inj}}$ is right proper and the forgetful functor $$\iota: \sMod^{\mathrm{inj}}_{\mathcal{R}} \to \sPSh_{\mathrm{inj}}$$ is a right Quillen functor. The proof that it also left proper is similar with the arguments above based on Boolean localization. 

It remains to show that the model category is also simplicial. Let $i: K \hookrightarrow L$ be an inclusion of simplicial sets and $p: M \to N$ a fibration in $\sMod^{\mathrm{inj}}_{\mathcal{R}}$. Since $p$ is also a fibration in $\sMod^{\mathrm{global}}_{\mathcal{R}}$, which is a simplicial model category by \cite[Lemma 2.2]{Ja3}, the canonical map $$ M^L \to N^L \underset{N^K}{\times} M^K$$ is a Kan fibration and it is trivial if either $i$ or $p$ is trivial. This concludes the proof of the theorem.
\end{proof}

\begin{remark} \label{proj}
It is clear from the proof that for every set of monomorphisms $\mathrm{I}$ in $\sMod_{\mathcal{R}}$ such that $\mathrm{I - inj} \subseteq \mathcal{W}_{\mathcal{R}}$, there is a left proper, simplicial, cofibrantly generated model category structure on $\sMod_{\mathcal{R}}$ with class of cofibrations $\mathrm{Cof(I)}$ and weak equivalences $\mathcal{W}_{\mathcal{R}}$. It is also right proper if
$\mathcal{R}\{i\}: \mathcal{R}\{X\} \to \mathcal{R}\{Y \}$ is in $\mathrm{Cof(I)}$ for every projective cofibration $i: X \to Y$ of simplicial presheaves.
\end{remark} 

The following proposition is now obvious. 

\begin{proposition}
The forgetful functor $\Phi: \sCoalg_{\mathcal{R}} \to \sMod^{\mathrm{inj}}_{\mathcal{R}}$ is a left Quillen functor.
\end{proposition}

\section{Proof of Theorem B} \label{ThmB}

\subsection{} We remind the reader of certain facts about the structure of coalgebras over a perfect field. For more details, see \cite{G}, \cite{Swe}.

Let $\mathbb{F}$ be a perfect field. An $\mathbb{F}$-coalgebra is called \textit{simple} if it has no non-trivial subcoalgebras. Every simple $\mathbb{F}$-coalgebra is finite dimensional and the dual $\mathbb{F}$-algebra is a finite field extension of $\mathbb{F}$. The \'{e}tale part $\acute{E}t(\textit{A})$ of an $\mathbb{F}$-coalgebra $A$ is the sum of all the simple subcoalgebras of $A$. This sum is known to be direct, see \cite[p. 166]{Swe}. According to the \textit{decomposition theorem} (see \cite[p.42]{Die}, \cite{G}), the inclusion 
\begin{equation*} 
\acute{E}t(A) \subseteq A
\end{equation*}
is a natural split monomorphism of coalgebras.  

If $\mathbb{F}$ is algebraically closed, then $\mathbb{F}$ is the unique simple $\mathbb{F}$-coalgebra up to isomorphism. The \'{e}tale part of an $\mathbb{F}$-coalgebra $A$ in this case can be identified with the 
canonical counit map $\mathbb{F}\{\rho(A)\} \to A$, that is, there is a natural isomorphism 
\begin{equation} \label{etale}
\mathbb{F}\{\rho(A)\} \stackrel{\cong}{\to} \acute{E}t(\textit{A}).
\end{equation}
Note of course that an arbitrary change of fields $\mathbb{F} \subseteq \mathbb{K}$ may give rise to $\mathbb{K}$-points of $A \otimes_{\mathbb{F}} \mathbb{K}$ that are not induced by $\mathbb{F}$-points 
of $A$. But any $\mathbb{K}$-point of $A \otimes_{\mathbb{F}} \mathbb{K}$ is already an 
$\overline{\mathbb{F}}$-point of $A \otimes_{\mathbb{F}} \overline{\mathbb{F}}$ where $\mathbb{F} \subseteq \overline{\mathbb{F}}$ denotes the algebraic closure (see \cite[Section 3]{Pa}). Therefore the isomorphism \eqref{etale} is natural with respect to field extensions of algebraically closed fields, that is, if $\mathbb{F} \subseteq \mathbb{K}$ are algebraically closed fields and $A$ is an $\mathbb{F}$-coalgebra, then there is a natural bijection $\rho(A) \stackrel{\cong}{\to} \rho(A \otimes_{\mathbb{F}} \mathbb{K})$.  

The isomorphism \eqref{etale} can be extended to a description of the \'{e}tale part of the $\mathbb{F}$-coalgebra $A$ in the general case where $\mathbb{F}$ is a perfect field. This is essentially a consequence of Galois theory. Let $\overline{\mathbb{F}}$ be the algebraic closure of $\mathbb{F}$ and $G$ the Galois group. The Galois group $G$ is regarded as a profinite group, so a $G$-action is always understood to be continuous. Recall that a $G$-action is continuous if and only if every element has a finite orbit. Let $\overline{A}:=A \otimes_{\mathbb{F}} \overline{\mathbb{F}}$ denote the associated $\overline{\mathbb{F}}$-coalgebra. The set of $\overline{\mathbb{F}}$-points of $\overline{A}$ generate the \'{e}tale part of $\overline{A}$ by \eqref{etale} above. Moreover, it is naturally a $G$-set. More explicitly, the $G$-action is defined as follows: given an $\overline{\mathbb{F}}$-point $f: \overline{\mathbb{F}} \to \overline{A}$ and $g \in G$, then $$(g \cdot f)(x) = (1 \otimes g)fg^{-1}(x).$$
In other words, there is a commutative diagram 
\begin{displaymath}
 \xymatrix{
\overline{\mathbb{F}} \ar[r]^f \ar[d]^g & \overline{A} \ar[d]^{1 \otimes g} \\
\overline{\mathbb{F}} \ar[r]^{g \cdot f} & \overline{A}
}
\end{displaymath}
The associated $\overline{\mathbb{F}}$-coalgebra $\overline{\mathbb{F}}\{\rho(\overline{A})\}$ is also naturally endowed with a $G$-action. The $G$-action is defined by the formula 
$$g(\sum x_i f_i) = \sum g(x_i) (g \cdot f_i).$$
This makes the canonical evaluation map 
$$\overline{\mathbb{F}}\{\rho(\overline{A})\} \to \overline{A}$$
$$\sum x_i f_i \mapsto \sum f_i(x_i)$$
invariant under the $G$-action. The \'{e}tale part $\acute{E}t(A)$ of $A$ is naturally isomorphic to the $G$-invariants of $\overline{\mathbb{F}}\{\rho(\overline{A})\}$, i.e., there is a natural isomorphism (see \cite[Proposition 2.8]{G}) \begin{equation} \label{etale2}
\overline{\mathbb{F}}\{\rho(\overline{A})\}^G \cong \acute{E}t(A).                                                    \end{equation}

More generally, if $X$ is a $G$-set, the $G$-invariants of the $\overline{\mathbb{F}}$-coalgebra $\overline{\mathbb{F}}\{X\}$ form naturally an $\mathbb{F}$-coalgebra. Let $\mathrm{Set}(G)$ denote the category of sets with a continuous $G$-action. This is a Grothendieck topos, e.g. see \cite[p. 596]{MM}. There is a well-defined functor $\overline{\mathbb{F}}\{-\}^G: \mathrm{Set}(G) \to \Coalg_{\mathbb{F}}$. This has a right adjoint that is defined
on objects by the formula $$\rho_G(A) = \Coalg_{\overline{\mathbb{F}}} (\overline{\mathbb{F}}, A \otimes_{\mathbb{F}} \overline{\mathbb{F}}).$$

\subsection{} We can now prove Theorem B. Let $\mathscr{F}$ be a presheaf of algebraically closed fields on $\mathcal{C}$. The main argument of the proof is in the following proposition.

\begin{proposition} \label{invariance-etale-splitting}
The functor $\rho: \sCoalg_{\mathscr{F}} \to \sPSh_{\mathrm{inj}}$ sends weak equivalences to $\mathscr{F}$-homology equivalences.
\end{proposition}
\begin{proof}
Let $f: A \to B$ be a weak equivalence of simplicial $\mathscr{F}$-coalgebras. The naturality of the splitting of the \'{e}tale part of coalgebras is respected along extensions of algebraically closed fields, so it follows that the map $$\mathscr{F}\{\rho(f)\}: \mathscr{F}\{\rho(A)\} \to \mathscr{F}\{ \rho(B)\}$$ is a retract of $f$ in the category of morphisms between simplicial $\mathscr{F}$-coalgebras. Since weak equivalences are closed under retracts, the required result follows. 
\end{proof}

To finish the proof of Theorem B, it suffices to show that the natural derived unit map $$X \to \mathbb{R}\rho (\mathscr{F}\{X\})$$ is a natural isomorphism in the $\mathscr{F}$-local homotopy category. This follows directly from the fact that both functors of the Quillen adjunction 
\begin{equation} \label{Q-adjunction3}
\mathscr{F}\{-\}: \mathrm{L}_{\mathscr{F}}\sPSh_{\mathrm{inj}} \rightleftarrows \sCoalg_{\mathscr{F}}: \rho
\end{equation}
preserve the weak equivalences, so the derived unit transformation is induced by the unit transformation of the coreflection \eqref{Q-adjunction3}. More explicitly, let $\mathscr{F}\{X\} \stackrel{\sim}{\to} \mathscr{F}\{X\}^{f}$ be a functorial fibrant replacement in $\sCoalg_{\mathscr{F}}$. The derived unit map of the Quillen adjunction can be represented by the natural map $$X \cong \rho (\mathscr{F} \{X \}) \to \rho (\mathscr{F}\{X\}^f).$$
By Proposition \ref{invariance-etale-splitting}, this is an $\mathscr{F}$-homology equivalence, so an isomorphism in the homotopy category of $\mathrm{L}_{\mathscr{F}} \sPSh_{\mathrm{inj}}$. This completes the proof of Theorem B.

The following is an immediate corollary.

\begin{corollary}
Let $X$ and $Y$ be simplicial presheaves in $\sPSh_{\mathrm{inj}}$ and $\mathscr{F}$ be a presheaf of algebraically closed fields. Then $X \cong Y$ in $\mathrm{Ho}(\mathrm{L}_{\mathscr{F}} \sPSh_{\mathrm{inj}})$ if and only if $\mathscr{F}\{ X \} \cong \mathscr{F}\{Y\}$ in $\mathrm{Ho}(\sCoalg_{\mathscr{F}})$.
\end{corollary}

\subsection{}
Let $\mathscr{F}$ be the constant presheaf at a perfect field $\mathbb{F}$. Theorem B together with the isomorphism \eqref{etale2} can be used to give a nice description of the derived unit transformation of \eqref{Q-adjunction3} in this case. The main idea is again based on the natural splitting of the \'{e}tale part of an $\mathbb{F}$-coalgebra, but now this can be related to the $G$-invariants of the $\overline{\mathbb{F}}$-points rather than with the $\mathbb{F}$-points directly. This brings the action of the Galois group into the picture. The arguments are completely analogous to \cite[pp.541-543]{G}, so we only sketch the necessary details. 

Let $\overline{\mathbb{F}}$ denote the algebraic closure of $\mathbb{F}$ and $G$ the profinite Galois group. Let $\sPShG$ denote the category of simplicial presheaves of $G$-sets.  Say that a morphism between simplicial presheaves of $G$-sets is a \textit{local weak equivalence} (resp. \textit{cofibration}) if the morphism of the underlying simplicial presheaves (of sets), by forgetting the $G$-action, is a local weak equivalence (resp. monomorphism). 

\begin{theorem}
The category $\sPShG$ together with the classes of local weak equivalences and cofibrations define a combinatorial model category.
\end{theorem}
\begin{proof}
The category $\sPShG$ can be equivalently viewed as the category of presheaves of simplicial $G$-sets. The category of simplicial $G$-sets, denoted by $\sSet(G)$, has a combinatorial model category structure where the cofibrations are the monomorphisms and the weak equivalences are the weak equivalences of the underlying simplicial sets \cite{G2}. Then there is an injective model category structure on $\sPShG$ where the cofibrations and the weak equivalences are defined sectionwise. This is again a combinatorial model category, see \cite[Proposition A.2.8.2]{Lu}. The required model category will be obtained as a left Bousfield localization of this injective model category by an application of Theorem \ref{j.smith}. We check that the conditions are satisfied: (i) and (ii) are obvious. Let $i: \ast \to G$ denote the obvious inclusion and $i^*: \sPShG \to \sPSh$ be the forgetful functor. Condition (iii) holds because $i^*$ preserves monomorphisms and pushouts. For (iv), note that the class of local weak equivalences in $\sPShG^{\to}$ is the inverse image of the class of local weak equivalences in $\sPSh$ under the accessible functor $i^*: \sPShG \to \sPSh$. Since the class of local weak equivalences in $\sPSh$ is accessible and accessibly embedded (by Theorem \ref{j.smith}), so is also the class of local weak equivalences in $\sPShG^{\to}$ by Proposition \ref{AR250}. Hence the conditions of Theorem \ref{j.smith} are satisfied and so the result follows.
\end{proof}

This model category will be denoted by $\sPShG_{\mathrm{inj}}$. Let $\overline{\mathscr{F}}$ denote the constant presheaf at $\overline{\mathbb{F}}$.  The next proposition shows that the comparison Quillen adjunction between $\mathrm{L}_{\mathscr{F}} \sPSh_{\mathrm{inj}}$ and $\sCoalg_{\mathscr{F}}$ factors through the model category $\sPShG_{\mathrm{inj}}$.

\begin{proposition} \label{G-Quillen-adj}
There are Quillen adjunctions
$$p^*: \sPSh_{\mathrm{inj}} \rightleftarrows \sPShG_{\mathrm{inj}}: (-)^G $$  
$$\overline{\mathscr{F}}\{-\}^G : \sPShG_{\mathrm{inj}} \rightleftarrows \sCoalg_{\mathscr{F}} : \rho_G.$$
\end{proposition}
\begin{proof}
The left adjoint $p^*$ is the pullback functor induced by the unique functor $G \to \ast$, i.e., $p^*(X)$ is the simplicial presheaf $X$ endowed with the trivial $G$-action. It is clear that $p^*$ preserves cofibrations and trivial cofibrations, so it is a left Quillen functor. The right adjoint is the limit functor which, in this case, is just the functor of $G$-fixed points. 

For a simplicial presheaf of $G$-sets $X$, the simplicial $\mathscr{F}$-coalgebra $\overline{\mathscr{F}}\{X\}^G$ is defined sectionwise by the formula $$\overline{\mathscr{F}}\{X\}^G(U)_n = \overline{\mathbb{F}}\{X(U)_n\}^G.$$ The right adjoint $\rho_G$ is defined by the formula $$\rho_G(A)(U)_n = \Coalg_{\overline{\mathbb{F}}} (\overline{\mathbb{F}}, A(U)_n \otimes_{\mathbb{F}} \overline{\mathbb{F}}).$$
$\overline{\mathscr{F}}\{-\}^G$ clearly preserves cofibrations. Moreover, there is a natural isomorphism (see the proof of \cite[Lemma 4.3]{G}),
\begin{equation}\label{an-iso} 
\overline{\mathscr{F}} \otimes_{\mathscr{F}} \overline{\mathscr{F}}\{X\}^G \stackrel{\cong}{\to} \overline{\mathscr{F}}\{X\}
\end{equation}
from which it follows that $\overline{\mathscr{F}}\{-\}^G$ sends $\overline{\mathscr{F}}$-homology equivalences to weak equivalences. In particular, it preserves trivial cofibrations and so it is a left Quillen functor.  
\end{proof}

Let $\mathrm{L}_{\overline{\mathscr{F}}} \sPShG_{\mathrm{inj}}$ denote the left Bousfield localization of $\sPShG_{\mathrm{inj}}$ at the class of $\overline{\mathscr{F}}$-homology equivalences of the underlying simplicial presheaves. The proof that this Bousfield localization exists is similar to the proof of Theorem \ref{Bousfield}. Moreover, there is an induced Quillen adjunction (cf. \cite[Lemma 4.3]{G}),
\begin{equation}\label{Galoiscontext}
\overline{\mathscr{F}}\{-\}^G : \mathrm{L}_{\overline{\mathscr{F}}}\sPShG_{\mathrm{inj}} \rightleftarrows \sCoalg_{\mathscr{F}} : \rho_G
\end{equation}
and an associated derived adjunction (cf. \cite[Proposition 4.4]{G}),
\begin{equation*}\label{Galoiscontext2}
\mathbb{L} \overline{\mathscr{F}}\{-\}^G : \mathrm{Ho}(\mathrm{L}_{\overline{\mathscr{F}}}\sPShG_{\mathrm{inj}}) \rightleftarrows \mathrm{Ho}(\sCoalg_{\mathscr{F}}) : \mathbb{R} \rho_G.
\end{equation*}
 
Note that the natural isomorphism \eqref{an-iso} shows that the adjunction \eqref{Galoiscontext} is a coreflection, i.e., the unit transformation of the adjunction is a natural isomorphism. The following theorem is the analogue of Theorem B for the presheaf $\mathscr{F}$.

\begin{theorem} \label{fullyfaithful2}
The functor $\mathbb{L} \overline{\mathscr{F}}\{-\}^G : \mathrm{Ho}(\mathrm{L}_{\overline{\mathscr{F}}}\sPShG_{\mathrm{inj}}) \rightarrow \mathrm{Ho}(\sCoalg_{\mathscr{F}})$ is fully faithful.
\end{theorem}
\begin{proof}
Similarly to the proof of Theorem B, it suffices to show that $\rho_G$ preserves the weak equivalences, i.e., it sends weak equivalences to $\overline{\mathscr{F}}$-homology equivalences. This is a consequence of the natural splitting of the \'{e}tale part of an $\mathbb{F}$-coalgebra similarly to the proof of Proposition \ref{invariance-etale-splitting}. The result follows from the identification of the \'{e}tale part by isomophism \eqref{etale2} and the isomorphism \eqref{an-iso}.
\end{proof}

The last theorem can be used to give a nice description of the unit transformation of the derived adjunction $$ \mathbb{L} \mathscr{F}\{-\} : \mathrm{L}_{\mathscr{F}}\sPSh_{\mathrm{inj}} \rightleftarrows \sCoalg_{\mathscr{F}} : \mathbb{R} \rho.$$
Let $X$ be a simplicial presheaf and $p^*(X)$ be the simplicial presheaf $X$ endowed with the trivial $G$-action. Let $p^*(X) \stackrel{\sim}{\to} p^*(X)^f$ be a functorial fibrant replacement of $p^*(X)$ in $\mathrm{L}_{\overline{\mathscr{F}}} \sPShG_{\mathrm{inj}}$. We have the following corollary. 

\begin{corollary} \label{derived-unit-perfect-case}
The canonical derived unit map $X \to (\mathbb{R}\rho)(\mathscr{F}\{X\})$ can be identified, up to a natural isomorphism in the homotopy category, with the map $X \to (p^*(X)^f)^G$.
\end{corollary}
\begin{proof}
Note that the Quillen adjunction $(\mathscr{F}\{-\},\rho)$ is the composition of the Quillen adjunctions of Proposition \ref{G-Quillen-adj}. By Theorem \ref{fullyfaithful2}, the derived unit transformation of the 
Quillen adjunction $(\overline{\mathscr{F}}\{-\}^G , \rho_G)$ is a natural isomorphism. Hence the result follows.  
\end{proof}

\subsection{} We end with a remark about the general case of an arbitrary presheaf $\mathscr{F}$ of perfect fields. The non-functoriality of algebraic closures becomes the main issue in treating this case using similar arguments. On the other hand, note that the class of $\mathscr{F}$-homology equivalences depends only on the characteristics of the fields involved.

We only comment on the following special case. Suppose that the Grothendieck site $\mathcal{C}$ has a terminal object denoted by 1. Examples include the site of open subsets of a topological space. Let $\mathscr{F}$ be an arbitrary presheaf of perfect fields on $\mathcal{C}$ and let $\mathscr{F}_1$ denote the constant presheaf at $\mathscr{F}(1)=\mathbb{F}$. Thus there is a morphism of presheaves $\mathscr{F}_1 \to \mathscr{F}$. Let $\overline{\mathbb{F}}$ be the algebraic closure of $\mathbb{F}$, $G$ the profinite Galois group and $\overline{\mathscr{F}}_1$ the constant presheaf at $\overline{\mathbb{F}}$. 

The Quillen adjunction 
\begin{equation} \label{composite-Q-adj}
\mathscr{F}_1\{-\}: \mathrm{L}_{\mathscr{F}} \sPSh_{\mathrm{inj}} \rightleftarrows \sCoalg_{\mathscr{F}_1}: \rho
\end{equation}
can written as the composition of the following three adjunctions: the Quillen equivalence
$$1: \mathrm{L}_{\mathscr{F}}\sPSh_{\mathrm{inj}} \to \mathrm{L}_{\mathscr{F}_1}\sPSh_{\mathrm{inj}}: 1$$
and the Quillen adjunctions  
$$p^*: \mathrm{L}_{\mathscr{F}_1} \sPSh_{\mathrm{inj}} \rightleftarrows \mathrm{L}_{\overline{\mathscr{F}}_1} \sPShG_{\mathrm{inj}}: (-)^G $$  
$$\overline{\mathscr{F}}_1\{-\}^G : \mathrm{L}_{\overline{\mathscr{F}}_1} \sPShG_{\mathrm{inj}} \rightleftarrows \sCoalg_{\mathscr{F}_1} : \rho_G.$$

Therefore the derived unit transformation of \eqref{composite-Q-adj} can be expressed in terms of the derived unit transformation of the Quillen adjunction $(\mathscr{F}_1\{-\}, \rho)$ as described in Corollary \ref{derived-unit-perfect-case}.

\end{document}